\font \sevenrm=cmr7
\font \fiverm=cmr5
\documentclass[12pt, english]{amsart}
\input epsf.sty
\usepackage{amsfonts}
\usepackage{amssymb}
\usepackage{amsmath}
\usepackage{amsthm}
\usepackage{amscd}
\usepackage[all]{xy} 
\usepackage{epsfig,exscale}
\usepackage{color}
\usepackage{graphicx}
\usepackage{xspace}
\usepackage{axodraw4j}
\usepackage{colordvi}

\newcommand{\nc}{\newcommand}

{\everymath{\displaystyle\everymath{}}\array}%
{\endarray}
\newtheorem{theorem}{Theorem}
\newtheorem{definition}{Definition}
\newtheorem{corollary}{Corollary}

\newtheorem{proposition}{Proposition}
\newtheorem{ex}{Example}
\newtheorem{remark}{Remark}
\nc{\comment}[1]{[[{\tt #1}]] }
\nc{\Cal}[1]{{\mathcal {#1}}}
\nc{\mop}[1]{\mathop{\hbox {\rm #1} }\nolimits}
\nc{\gmop}[1]{\mathop{\hbox {\bf #1} }\nolimits}

\nc{\smop}[1]{\mathop{\hbox {\sevenrm #1} }\nolimits}
\nc{\ssmop}[1]{\mathop{\hbox {\fiverm #1} }\nolimits}
\nc{\mopl}[1]{\mathop{\hbox {\rm #1} }\limits}
\def\dbar{d\hskip-3pt \raise 4pt\hbox{-}}
\nc{\smopl}[1]{\mathop{\hbox {\sevenrm #1} }\limits}
\nc{\ssmopl}[1]{\mathop{\hbox {\fiverm #1} }\limits}
\nc{\frakg}{{\frak g}}
\nc{\g}[1]{{\frak {#1}}}
\def \restr#1{\mathstrut_{\textstyle |}\raise-6pt\hbox{$\scriptstyle #1$}}
\def \srestr#1{\mathstrut_{\scriptstyle |}\hbox to
-1.5pt{}\raise-4pt\hbox{$\scriptscriptstyle #1$}}
\nc{\wt}{\widetilde} \nc{\wh}{\widehat}
\nc{\redtext}[1]{\textcolor{red}{#1}}
\nc{\bluetext}[1]{\textcolor{blue}{#1}}
\nc\fleche[1]{\mathop{\hbox to #1 mm{\rightarrowfill}}\limits}
\nc{\ignore}[1]{}
\def\semi{\mathrel{\times}\kern -.85pt\joinrel\mathrel{\raise
1.4pt\hbox{${\scriptscriptstyle |}$}}}
\nc\R{{\mathbb R}}
\nc\N{{\mathbb N}}
\nc\inver{^{-1}}
\nc\point{\hbox{\bf .}}
\nc\un{\hbox{\bf 1}}
\nc\ootimes{\overline\otimes}
\let\precrel\prec
\let\succrel\succ
\renewcommand{\prec}{\left.\precrel\right.}
\renewcommand{\succ}{\left.\succrel\right.}
\def\shu{\joinrel{\!\scriptstyle\amalg\hskip -3.1pt\amalg}\,}
\def\sshu{\scalebox{0.8}{$\shu$}}

\def\diagramme #1{\vskip 4mm \centerline {#1} \vskip 4mm}
\begin{document}
%%%%%%%%%%%%%%%%%%%%%%%%%%%%%%%%%%%%%%%%%%%%%
%%%%%%%%%%%%%%%%%%%%%%%%%%%%%%%%%%%%%%%%%%%%%%%%%%%%%%%%%%%%%%%%%%%%%%%%%%%%%%%%%%%%%%%%%%
%%%%%%%%%%%%%%%%%%%%%%%%%%%%%%%%%%%%%%%%%%%%%
%%%%%%%%%%%%%%%%%%%%%%% %%%%%%%%%%%%%%%%%%%%%%
%%%%%%%%%%%%%%%%%%%%%%%%%%%%%%%%%%%%%%%%%%%%%
\title{
{Shuffle quadri-algebras and concatenation}}        
\author{Mohamed Belhaj Mohamed}
\address{{Mathematics Departement, Sciences college, Taibah University, Kingdom of Saudi Arabia~.}\vspace{0.01cm}
{Laboratoire de math\'ematiques physique fonctions sp\'eciales et applications, Universit\'e de Sousse, rue Lamine Abassi 4011 H. Sousse,  Tunisie.}}     
         \email{mohamed.belhajmohamed@isimg.tn}
 
\author{Dominique Manchon}
\address{Universit\'e Blaise Pascal,
       C.N.R.S.-UMR 6620,
        63177 Aubi\`ere, France}       
      \email{manchon@math.univ-bpclermont.fr}
        \urladdr{http://math.univ-bpclermont.fr/~manchon/}     
%%%%%%%%%%%%%%%%%%%%%%%%%%%%%%%%%%%%%%%%%%%%%%%%%%%%%%%%%%%%%%%%%%%
\date{June 2018}
\noindent{\footnotesize{${}\phantom{a}$ }}
%%%%%%%%%%%%%%%%%%%%%%%%%%%%%%%%%%%%%%%%%%%%%%%%%%%%%%%%%%%%%%%%%%%
%%%%%%%%%%%%%%%%%%%%%%%%%%%%%%%%%%%%%%%%%%%%%%%%%%%%%%%%%%%%%%%%%%%
\begin{abstract}
In this article, we study the shuffle quadri-algebra over some vector space. We prove the existence of some relations between the four quadri-algebra laws which constitute the shuffle product, the concatenation product and the deconcatenation coproduct. We also show that the shuffle quadri-algebra admits two module-algebra structures on itself endowed with the underlying associative algebra structure.
\end{abstract}
%\newpage 
\maketitle
\textbf{MSC Classification}: 05E40, 17A30.

\textbf{Keywords}: quadri-algebra, dendriform algebra, shuffle, concatenation, module-algebra.
\tableofcontents
%%%%%%%%%%%%%%%%%%%%%%%%%%%%%%%%%%%%%%%%%%%%%%%%%%%%%%%%%%%%%%%%%%%
%%%%%%%%%%%%%%%%%%%%%%%%%%%%%%%%%%%%%%%%%%%%%%%%%%%%%%%%%%%%%%%%%%%
%%%%%%%%%%%%%%%%%%%%%%%%%%%%%%%%%%%%%%%%%%%%%%%%%%%%%%%%%%%%%%%%%%%

%%%%%%%%%%%%%%%%%%%%%%%%%%%%%%%%%%%%%%%%%%%%%%%%%%%%%%%%%%%%%%%%
%%%%%%%%%%%%%%%%%%%%%%%%%%%%%%%%%%%%%%%%%%%%%%%%%%%%%%%%%%%%%%%%
\section{Introduction}
A dendriform algebra is a vector space equipped with an associative product which can be written as a sum of two operations   $\prec$ and $\succ$ called left and right respectively, which satisfy the three following rules:
\begin{eqnarray*}
 (x \prec y) \prec z &=& x \prec (y \prec z) + x \prec (y \succ z)\nonumber\\
(x \succ y) \prec z &=& x \succ (y \prec z)\\
(x \prec y) \succ z + (x \succ y) \succ z &=& x \succ(y \succ z) \nonumber
\end{eqnarray*} 

They were introduced by Jean-Louis Loday \cite[\S5]{jll} in $1995$ with motivation from algebraic K-theory and have been studied by other authors in different domains \cite{ma, kd, kudo, emp, jllr, lr, mr}.\\

In 2004, Marcelo Aguiar and Jean-Louis Loday introduced the notion of quadri-algebra in \cite{al}. A quadri-algebra is an associative algebra the multiplication of which can be decomposed as the sum of four operations $\searrow$, $\nearrow$, $\nwarrow$ and $\swarrow$ satisfying nine axioms. Two dendriform structures are attached to a quadri-algebra: the first dendriform structure is given by the two operations  $\succ$ and $\prec$ such that:
 \begin{eqnarray*}
 x \succ y &:=& x \nearrow y + x \searrow y\\ 
x \prec y &:=& x \nwarrow y + x \swarrow y,
\end{eqnarray*}
and the second is given by the two operations $\vee$ and $\wedge$ where: 
 \begin{eqnarray*}
x \vee y &:=&  x \searrow y + x \swarrow y \\
x \wedge y &:=&  x \nearrow y + x \nwarrow y. 
\end{eqnarray*}
Quadri-algebras were studied by Lo\"\i c Foissy  together with quadri-coalgebras and quadri-bialgebras \cite{folo}.\\

In this article we revisit the canonical example of shuffle quadri-algebra \cite{ght, fp, lr, re} treated by Marcelo Aguiar, Jean-Louis Loday and Lo\"\i c Foissy. We prove that there exist relations between the quadri-algebra laws, the concatenation product and the deconcatenation coproduct. We show that, for any elements $u, v, w$ in this quadri-algebra $\mathcal H$ we have:
\begin{eqnarray*}
 u \nearrow (v w) &=& \sum_{u = u^1 u^2} (u^1 \searrow v) ( u^2 \wedge w) \\
&=& \sum_{u = u^1 u^2} ( u^1 \succ v)  ( u^2 \nearrow w), 
\end{eqnarray*}
\begin{eqnarray*}
 u \searrow (v w) &=& \sum_{u = u^1 u^2} (u^1 \searrow v) ( u^2 \vee w) \\
&=& \sum_{u = u^1 u^2} ( u^1 \succ v)  ( u^2 \searrow w), 
\end{eqnarray*}
\begin{eqnarray*}
 u \swarrow (v w) &=& \sum_{u = u^1 u^2} (u^1 \swarrow v) ( u^2 \vee w) \\
&=& \sum_{u = u^1 u^2} ( u^1 \prec v)  ( u^2 \searrow w), 
\end{eqnarray*}
\begin{eqnarray*}
 u \nwarrow (v w) &=& \sum_{u = u^1 u^2} (u^1 \swarrow v) ( u^2 \wedge w) \\
&=& \sum_{u = u^1 u^2} ( u^1 \prec v)  ( u^2 \nearrow w)
\end{eqnarray*}
whenever these expressions make sense. We derive from these results a set of relations between the dendriform laws, the concatenation and the deconcatenation coproduct. We show that, for any $u, v, w \in \Cal H$, we have:
\begin{eqnarray*}
 u  \wedge ( v w) &=& u \nearrow (v w) +  u \nwarrow (v w)\\
 &=& \sum_{u = u^1 u^2} (u^1 \vee v) ( u^2 \wedge w), 
\end{eqnarray*}
\begin{eqnarray*}
u  \prec ( v w)  &=& u \nwarrow (v w) + u \swarrow (v w)\\
 &=& \sum_{u = u^1 u^2} (u^1 \prec v) ( u^2 \succ w),
\end{eqnarray*}
\begin{eqnarray*}
u  \vee ( v w) &=& u \swarrow (v w) + u \searrow (v w)  \\
&=& \sum_{u = u^1 u^2} (u^1 \vee v) ( u^2 \vee w),
\end{eqnarray*}
\begin{eqnarray*}
 u  \succ ( v w)  &=& u \searrow (v w) +  u \nearrow (v w) \\
&=& \sum_{u = u^1 u^2} (u^1 \succ v) ( u^2 \succ w),
\end{eqnarray*}
and consequently, two relations between the shuffle product, the concatenation and the deconcatenation coproduct. We show that, for any $u, v, w \in \Cal H$, we have:
\begin{eqnarray*}
 u \shu (v w) &=& \sum_{u = u^1 u^2} (u^1 \vee v) ( u^2 \shu w) \\
&=& \sum_{u = u^1 u^2} ( u^1 \shu v)  ( u^2 \succ w).
\end{eqnarray*}
At the end of this article, we prove the existence of two module-algebra structures on $\Cal H$ given by $\vee$ and $\succ$, and further compatibility relations. All these results are best expressed in terms of commutative diagrams involving an extended version of the tensor product.\\

\noindent \textbf{Acknowledgements}. We thank Lo\"\i c Foissy for his useful remarks, and the referee for a careful reading and numerous suggestions which greatly helped us to improve the presentation.
%%%%%%%%
\section{A modified tensor product}
%%%%%%%%
Let $V$ and $W$ be two vector spaces over a field $k$, let $\varepsilon_V:V\oplus k\to\hskip -3.5mm\to k$ be the projection onto the second component, and let us consider $\varepsilon_W$ similarly. We set:
\begin{equation}
V\ootimes W:=\mop{Ker}(\varepsilon_V\otimes\varepsilon_W)\simeq (V\otimes W)\oplus V\oplus W.
\end{equation}
This modified tensor product is symmetric and associative. We denote by $v+\lambda\un$ the generic element of $V\oplus k$. The generic element of $V\ootimes W$ can be written as
$$\sum_{k=1}^n v_k\otimes w_k+v\otimes\un + \un\otimes w,$$
where $v,v_1,\ldots,v_n\in V$ and $w,w_1,\ldots,w_n\in W$. For any pair of linear maps $f:V\to V'$ and $g:W\to W'$, there is a unique linear map $f\ootimes g:V\ootimes W\to V'\ootimes W'$ such that for any $v\in V$ and $w\in W$ we get
$$f\ootimes g(v\otimes w)=f(v)\otimes g(w),\hskip 8mm f\ootimes g(v\otimes\un)=f(v)\otimes\un,
\hskip 8mm f\ootimes g(\un\otimes w)=\un\otimes g(w).$$
These data turn the collection of $k$-vector spaces into a symmetric monoidal category. The unit for this tensor product $\ootimes$ is the zero-dimensional space $\{0\}$.\\

Recall that any associative algebra $V$ gives rise to a unital associative algebra $\overline V=V\oplus k\un$. As a consequence, the product $m:V\otimes V\to V$ is extended to a product from $\overline V\otimes \overline V$ into $\overline V$, and its restriction to $V\ootimes V$ takes value in $V$. Associativity of this extension can be described by the commutativity of the following diagram:
\diagramme{
\xymatrix{V\ootimes V\ootimes V\ar[rr]^{m\ootimes I}\ar[d]_{I\ootimes m} &&V\ootimes V\ar[d]^m\\
V\ootimes V\ar[rr]_m && V
}
}
In the same line of thought, a left module structure $\Phi:V\otimes M\to M$ yields an extension $\Phi:V\ootimes M\to M$ via
$$\Phi(\un\otimes m)=m,\hskip 12mm \Phi(v\otimes\un)=0,$$
making the following diagram commute:
\diagramme{
\xymatrix{V\ootimes V\ootimes M\ar[rr]^{m\ootimes I}\ar[d]_{I\ootimes \Phi} &&V\ootimes M\ar[d]^\Phi\\
V\ootimes M\ar[rr]_\Phi && M
}
}

and similarly for right modules. Dually, a coassociative coproduct $\wt\Delta:V\to V\otimes V$ can be modified to a co-unital coproduct $\Delta:\overline V\to \overline V\otimes \overline V$. Its restriction $\Delta:V\to V\ootimes V$ makes the following co-associativity diagram commute:
\diagramme{
\xymatrix{V\ar[rr]^\Delta\ar[d]_\Delta && V\ootimes V\ar[d]_{\Delta\ootimes I}\\
V\ootimes V\ar[rr]_{I\ootimes \Delta}&&V\ootimes V\ootimes V
}
}
Commutative diagrams for extended left and right comodule structures involving $\ootimes$ can be drawn accordingly. 
%%%%%%%%%%%%%%%%%%%%%%%%%%%%%%%%%%%%%%%%%%%%%%%%%%%%%%%%%%%%%%%%%%%
\section{Dendriform algebras}
%%%%%%%%%%%%%%%%%%%%%%%%%%%%%%%%%%%%%%%%%%%%%%%%%%%%%%%%%%%%%%%%%%%%%%%%%%%%%%%
%%%%%%%%%%%%%%%%%%%%%%%%%%%%%%%%%%%%%%%%%%%%%%%%%%%%%%%%%%%%%%%%%%%%%%%%%%%%%%%
%%%%%%%%%%%%%%%%%%%%%%%%%%%%%%%%%%%%%%%%%%%%%%%%%%%%%%%%%%%%%%%%%%%
%%%%%%%%%%%%%%%%%%%%%%%%%%%%%%%%%%%%%%%%%%%%%%%%%%%%%%%%%%%%%%%%%%%

A dendriform algebra is a vector space $D$ together with two operations $\prec : D \otimes D \longrightarrow D$ and $\succ : D \otimes D \longrightarrow D$, called left and right respectively, such that:
\begin{eqnarray}\label{dend}
 (x \prec y) \prec z &=& x \prec (y \prec z) + x \prec (y \succ z)\nonumber\\
(x \succ y) \prec z &=& x \succ (y \prec z)\\
(x \prec y) \succ z + (x \succ y) \succ z &=& x \succ(y \succ z). \nonumber
\end{eqnarray} 
Dendriform algebras were introduced in \cite[\S5]{jll}. See also \cite{ma, kd, kudo, emp, jllr, lr, mr} for additional work on this subject. Defining a new operation $\star$ by:
\begin{equation}\label{star}
x \star y  := x \prec y + x \succ y
\end{equation} 
permits us to rewrite axioms \eqref{dend} as:
\begin{eqnarray}\label{axioms-dend-bis}
 (x \prec y) \prec z &=& x \prec (y \star z)\nonumber\\
(x \succ y) \prec z &=& x \succ (y \prec z)\\
(x \star y) \succ  z &=& x \succ(y \succ z)\nonumber. 
\end{eqnarray} 

By adding the three relations we see that the operation $\star$ is associative. For this reason, a dendriform algebra may be regarded as an associative algebra $(D, \star)$ for which the multiplication $\star$ can be decomposed as the sum of two coherent operations.\\

\noindent It is standard to extend the dendriform operations $\prec$ and $\succ$ to $D\ootimes D$ by setting:
\begin{equation}\label{dend-ext}
a  \succ \un = 0, \hskip 6mm  \un \succ a = a, \hskip 6mm\un \prec a = 0, \hskip 6mm  a \prec \un = a
\end{equation}
for any $a\in D$. The space $\overline D:=D\oplus k\un$ is called somewhat incorrectly a unital dendriform algebra, although $\un\prec\un$ and $\un\succ \un$ are not defined. Let us however remark that $\star=\prec+\succ$ can be extended to $\overline D\otimes\overline D$ , making $\overline D$ a unital associative algebra. The extension \eqref{dend-ext} is consistent with the dendriform axioms \eqref{axioms-dend-bis} in the sense that each of the three axioms makes sense provided both members of the equality are defined.

%%%%%%%%%%%%%%%%%%%%%%%%%%%%%%%%%%%%%%%%%%%%%%%%%%%%%%%%%%%%%%%%%%%
%%%%%%%%%%%%%%%%%%%%%%%%%%%%%%%%%%%%%%%%%%%%%%%%%%%%%%%%%%%%%%%%%%%
\section{Quadri-algebras}
%%%%%%%%%%%%%%%%%%%%%%%%%%%%%%%%%%%%%%%%%%%%%%%%%%%%%%%%%%%%%%%%%%%%%%%%%%%%%%%
%%%%%%%%%%%%%%%%%%%%%%%%%%%%%%%%%%%%%%%%%%%%%%%%%%%%%%%%%%%%%%%%%%%%%%%%%%%%%%%
%%%%%%%%%%%%%%%%%%%%%%%%%%%%%%%%%%%%%%%%%%%%%%%%%%%%%%%%%%%%%%%%%%%
%%%%%%%%%%%%%%%%%%%%%%%%%%%%%%%%%%%%%%%%%%%%%%%%%%%%%%%%%%%%%%%%%%%
In this section, we use definitions and results on quadri-algebra structures given by Marcelo Aguiar and Jean-Louis Loday in \cite{al} and Lo\"\i c Foissy in \cite{folo}. A quadri-algebra structure consists in splitting an associative product into four operations, which in turn gives rise to two distinct dendriform structures.
\begin{definition}
 A quadri-algebra is a vector space $Q$ together with four operations:
$$\searrow, \nearrow, \nwarrow \text{and} \swarrow : Q\otimes Q \longrightarrow Q,$$
satisfying the nine axioms below. In order to state them, consider the following operations:
\begin{eqnarray}
 x \succ y &:=& x \nearrow y + x \searrow y\\ 
x \prec y &:=& x \nwarrow y + x \swarrow y \\
x \vee y &:=&  x \searrow y + x \swarrow y \\
x \wedge y &:=&  x \nearrow y + x \nwarrow y 
\end{eqnarray} 
and:
\begin{eqnarray}\label{dra}
 x \star y &:=& x \nearrow y + x \searrow y + x \nwarrow y + x \swarrow y \nonumber\\
            &:=& x \succ y  + x \prec y\\ 
            &:=&  x \vee y + x \wedge y\nonumber . 
\end{eqnarray} 
The nine axioms, stated by Marcelo Aguiar and Jean-Louis Loday in \cite{al} are:
\begin{eqnarray*}
 (x \nwarrow y) \nwarrow z = x \nwarrow (y \star z), &\  (x \nearrow y) \nwarrow z = x \nearrow (y \prec z), &\  (x  \wedge y) \nearrow z = x \nearrow (y \succ z),\\
 (x \swarrow y) \nwarrow z = x \swarrow (y \wedge z), &\  (x \searrow y) \nwarrow z = x \searrow (y \nwarrow z), &\  (x  \vee y) \nearrow z = x \searrow (y \nearrow z), \\
 (x \prec y) \swarrow z = x \swarrow (y \vee z), &\  (x \succ y) \swarrow z = x \searrow (y  \swarrow z), &\  (x  \star y) \searrow z = x \searrow (y \searrow  z).
\end{eqnarray*}
\end{definition} 

The operations $\searrow, \nearrow, \nwarrow, \swarrow $ are referred to as southeast, northeast, northwest, and
southwest, respectively. Accordingly, $\wedge, \vee, \prec \text{and} \succ$ are called north, south, west and east respectively.
The axioms are displayed in the form of a $3\times3$ matrix. As in \cite{al}, we will make use of standard
matrix terminology (entries, rows and columns) to refer to them.\\

Let $Q$ be a quadri-algebra. Following \cite[Paragraph 3.1]{folo}, we extend the four products to $Q \ootimes Q$ in the following way: if $a \in  Q$,
\begin{eqnarray*} 
 a \nwarrow \un = a,\hskip 6mm a \nearrow \un = 0, && \un \nwarrow a = 0,\hskip 6mm  \un \nearrow a = 0, \\
 a \swarrow \un = 0 ,\hskip 6mm a \searrow \un = 0, && \un \swarrow a = 0,\hskip 6mm \un \searrow a = a.
\end{eqnarray*} 
It follows that we have for any $a \in Q$:
\begin{eqnarray*} 
 a \wedge \un = a,\hskip 6mm \un \wedge a = 0, && \un \vee a = a,\hskip 6mm  a \vee \un = 0, \\
 a  \succ \un = 0,\hskip 6mm \un \succ a = a, && \un \prec a = 0,\hskip 6mm  a \prec \un = a.
\end{eqnarray*} 

%%%%%%%%%%%%%%%%%%%%%%%%%%%%%%%%%%%%%%%%%%%%%%%%%%%%%%%%%%%%%%%%%%%
%%%%%%%%%%%%%%%%%%%%%%%%%%%%%%%%%%%%%%%%%%%%%%%%%%%%%%%%%%%%%%%%%%%
\section{from quadri-algebras to dendriform algebras}
%%%%%%%%%%%%%%%%%%%%%%%%%%%%%%%%%%%%%%%%%%%%%%%%%%%%%%%%%%%%%%%%%%%%%%%%%%%%%%%
%%%%%%%%%%%%%%%%%%%%%%%%%%%%%%%%%%%%%%%%%%%%%%%%%%%%%%%%%%%%%%%%%%%%%%%%%%%%%%%
%%%%%%%%%%%%%%%%%%%%%%%%%%%%%%%%%%%%%%%%%%%%%%%%%%%%%%%%%%%%%%%%%%%
\noindent The three column sums in the matrix of quadri-algebra axioms yield:
\begin{eqnarray*}
 (x \prec y) \prec z = x \prec (y \star z),\;\;\; (x \succ y) \prec z = x \succ (y \prec z)\;\; \text{and}\; \; (x \star y) \succ z = x \succ (y \succ z).   
\end{eqnarray*}
Thus, endowed with the operations west for left and east for right, $Q$ is a dendriform algebra. We denote it by $Q_h$ and call it the horizontal dendriform algebra associated to $Q$.
Considering instead the three row sums  in the matrix of quadri algebra axioms yields:
\begin{eqnarray*}
 (x \wedge y) \wedge  z = x \wedge (y \star z),\;\;\; (x \vee y) \wedge  z = x \vee (y \wedge  z)\;\; \text{and}\; \; (x \star y) \vee z = x \vee (y \vee z).
\end{eqnarray*}
Thus, endowed with the operations north for left and south for right, $Q$ is a dendriform algebra. We denote it by $Q_v$ and call it the vertical dendriform algebra associated to $Q$. The associative operations corresponding to the dendriform algebras $Q_h$ and $Q_v$ by means of \eqref{star} coincide, according to \eqref{dra}. In other words,
\begin{equation}
\star = \nearrow + \searrow + \swarrow + \nwarrow = \prec + \succ = \wedge + \vee.
\end{equation}
%%%%%%%%%%%%%%%%%%%%%%%%%%%%%%%%%%%%%%%%%%%%%%%%%%%%%%%%%%%%%%%%%%%
%%%%%%%%%%%%%%%%%%%%%%%%%%%%%%%%%%%%%%%%%%%%%%%%%%%%%%%%%%%%%%%%%%%%%%%%%%%%%%%
%%%%%%%%%%%%%%%%%%%%%%%%%%%%%%%%%%%%%%%%%%%%%%%%%%%%%%%%%%%
%%%%%%%%%%%%%%%%%%%%%%%%%%%%%%%%%%%%%%%%%%%%%%%%%%%%%%%%%%%%%%%%%%%
%%%%%%%%%%%%%%%%%%%%%%%%%%%%%%%%%%%%%%%%%%%%%%%%%%%%%%%%%%%%%%%%%%%
%%%%%%%%%%%%%%%%%%%%%%%%%%%%%%%%%%%%%%%%%%%%%%%%%%%%%%%%%%%%%%%%%%%
\section{The shuffle quadri-algebra}
%%%%%%%%%%%%%%%%%%%%%%%%%%%%%%%%%%%%%%%%%%%%%%%%%%%%%%%%%%%%%%%%%%%%%%%%%%%%%%%
%%%%%%%%%%%%%%%%%%%%%%%%%%%%%%%%%%%%%%%%%%%%%%%%%%%%%%%%%%%%%%%%%%%%%%%%%%%%%%%
%%%%%%%%%%%%%%%%%%%%%%%%%%%%%%%%%%%%%%%%%%%%%%%%%%%%%%%%%%%%%%%%%%%
Let $k$ be a field, and let $V$ be a $k$-vector space. Let $\Cal H = T (V) = \bigoplus_{n \geq 0} V^{\otimes n}$ be the tensor algebra of $V$, where we denote by $\Delta$ the deconcatenation coproduct and by $m$ the concatenation product. Let $\Cal H^+ = T^+ (V) = \bigoplus_{n \geq 1} V^{\otimes n}$ be the augmentation ideal. For all $u, v \in \Cal H$, we have:
\begin{eqnarray} 
m(u \otimes v) &=& uv,
\end{eqnarray}
and
\begin{eqnarray}
\Delta(u) &=&  \sum_{u = u^1 u^2} u^1 \otimes u^2.
\end{eqnarray}
Here $u^1$ (resp. $u^2$) is the left (resp. right) part of the word $u$ defined by the place where $u$ is cut. This notation matches with Sweedler's notation for a coproduct in a coalgebra in general. The shuffle product $\shu$ is defined for any $u = u_1 u_2 \ldots u_p$ and $v = u_{p + 1} u_{p + 2} \ldots u_{p + q}$ with $u_1,\ldots, u_{p+Q}\in V$ by:
\begin{equation} 
u \shu v = \sum_{\sigma \in \smop{Sh} (p, q)} u_{\sigma^{-1} (1)}  u_{\sigma^{-1} (2)} \ldots \ldots u_{\sigma^{-1} (p + q)},
\end{equation}
where $\mop{Sh}(p, q)$ denotes the set of $\sigma \in S_{p+q}$ verifying $\sigma(1) <\cdots< \sigma(p)$ and $\sigma(p+1) <\cdots< \sigma(p+q)$. The triple $(\Cal H , \shu, \Delta)$ becomes a commutative Hopf algebra called the shuffle Hopf algebra. The shuffle algebra of a vector space $V$ provides an example of a commutative quadri-algebra (see Remark \ref{comm} below). Let us recall the two recursive formulas for the shuffle product:
\begin{eqnarray*}
au\shu bv&=&a(u\shu bv)+b(au\shu v),\\
ua\shu vb&=&(u\shu vb)a+(ua\shu v)b
\end{eqnarray*}
for any $a,b\in V$ and $u,v\in T(V)$.\\

The quadri-algebra laws on $\Cal H^+$ are defined by Marcelo Aguiar and Jean-Louis Loday in \cite{al} recursively on the degrees of $u$ and $v$, and can be extended to $\Cal H^+\ootimes\Cal H^+$ as explained above. We recall here the construction.\begin{enumerate}
\item If $u = \un$ and $v \in \Cal H^+$, we have:
\begin{eqnarray*} 
\un \shu v &=& v
\end{eqnarray*}
and:
\begin{eqnarray*}
\un \nearrow v = 0,\hskip 5mm  \un \searrow v = v, &&\un \swarrow v = 0, \hskip 5mm\un \nwarrow v = 0,\\
v\swarrow \un=0,\hskip 5mm  v\nwarrow\un = v, && v\nearrow \un=0, \hskip 5mm v\searrow \un=0,
\end{eqnarray*}
\noindent which immediately gives:
\begin{eqnarray*}
\un \succ v =v\prec \un = v, && \un \prec v =v\succ \un= 0, \\
\un  \wedge v =v\vee\un = 0, && \un \vee v =v\wedge\un =  v.
\end{eqnarray*} 
\item If $u, v  \in V$, we have: 
\begin{eqnarray*} 
u \shu v &=& uv + vu
\end{eqnarray*}
\noindent and:
\begin{eqnarray*}
 u \nearrow v = vu, && u \searrow v = 0, \\
 u \swarrow v = uv, && u \nwarrow v = 0,
\end{eqnarray*}

\noindent which immediately gives:
\begin{eqnarray*}
 u \succ v = vu, &&  u \prec v = uv, \\
u  \wedge v = vu, && u \vee v =  uv.
\end{eqnarray*} 
\item If $u\in V$, and $v = c \theta d$ with $c,d\in V$ and $\theta\in V^{\otimes (n-2)}$, we have:
\begin{eqnarray*} 
u \shu v &=& u \shu c \theta d\\
&=& u c \theta d + c (u \shu \theta) d + c \theta d u  + 0.
\end{eqnarray*}

\noindent The four quadri-algebra laws on $\Cal H$ are given by:
\begin{eqnarray}
 u \nearrow v = c \theta d u, && u \searrow v = c (u \shu \theta) d, \nonumber\\
 u \swarrow v = u c \theta d, && u \nwarrow v = 0 \label{nul},
\end{eqnarray}

\noindent which immediately gives:
\begin{eqnarray*} 
u \prec v = u c \theta d, && u \succ v =  c (u \shu \theta) d + c \theta d u,\\
u  \wedge v =  c \theta d u, && u \vee v =   c (u \shu \theta) d + u c \theta d.
\end{eqnarray*} 
\item If $u, v \in\Cal H$, such that $u, v$ are pure tensors of of degree $\geq 2$, i.e,  $u = awb$ and $v = c \theta d$ with $a,b,c,d\in V$, we have:
\begin{eqnarray*} 
u \shu v  &=& a(wb \shu c\theta)d + c (awb \shu \theta)d + a(w \shu c \theta d) b + c(aw \shu \theta d) b.
\end{eqnarray*}

\noindent The four quadri-algebra operations on $\Cal H$ are defined by:
\begin{eqnarray*}
 u \nearrow v = c(aw \shu \theta d) b, && u \searrow v = c (awb \shu \theta)d, \\
 u \swarrow v = a(wb \shu c\theta)d, && u \nwarrow v = a(w \shu c \theta d) b.
\end{eqnarray*}

\noindent The dendriform algebra operations on $\Cal H$ are defined by:
\begin{eqnarray*}
u \succ v = c (awb \shu \theta d), && u \prec v = a(wb \shu c\theta d), \\
u  \wedge v =   (aw \shu c \theta d) b, && u \vee v =  (awb \shu c \theta)d. 
\end{eqnarray*} 
\end{enumerate}  

\noindent We verify easily then:
\begin{eqnarray}\label{shu}
u \shu v &=& u \nearrow v + u \searrow v + u \nwarrow v + u \swarrow v \nonumber\\
            &=& u \succ v  + u \prec v\\ 
            &=& u \vee v + u \wedge v\nonumber. 
\end{eqnarray} 

\noindent The nine axioms of quadri-algebra laws can now be easily verified. 
\begin{remark}\label{comm}\rm
By commutativity of the shuffle product, the quadri-algebra laws verify for any $u,v$ of length $\ge 2$:
 \begin{eqnarray*}
 u \nearrow v &=& c(aw \shu \theta d) b  \\
& = & c(\theta d \shu a w) b \\
& = & v \swarrow u.
\end{eqnarray*}
\begin{eqnarray*}
 u \searrow v &=& c (awb \shu \theta) d \\
&=& c (\theta \shu awb)d \\
&=& v \nwarrow u.
\end{eqnarray*}
The verification of these two commutativity statements for any $u,v$ such that $u\otimes v\in\Cal H^+\ootimes\Cal  H^+$ is straightforward and left to the reader. The shuffle quadri-algebra $\Cal H^+$ is said to be \textsl{commutative}.
\end{remark}
\begin{remark}\cite{folo} \rm
The four quadri-algebra operations also admit a non-recursive definition in terms of shuffles: supposing that $u$ (resp. $v$) is a word of length $p$ (resp. $q$),
\begin{eqnarray}
u \searrow v &=& \sum_{\sigma \in \smop{Sh} (p, q),\, \sigma^{-1}(1)= p+1 \smop{ and }\sigma^{-1}(p+q)= p+q} u_{\sigma^{-1} (1)}  u_{\sigma^{-1} (2)} \ldots \ldots u_{\sigma^{-1} (p + q)},\\
u \nearrow v &=& \sum_{\sigma \in \smop{Sh} (p, q),\, \sigma^{-1}(1)= p+1 \smop{ and }\sigma^{-1}(p+q)= p} u_{\sigma^{-1} (1)}  u_{\sigma^{-1} (2)} \ldots \ldots u_{\sigma^{-1} (p + q)},\\
u \nwarrow v &=& \sum_{\sigma \in \smop{Sh} (p, q),\, \sigma^{-1}(1)= 1 \smop{ and }\sigma^{-1}(p+q)= p} u_{\sigma^{-1} (1)}  u_{\sigma^{-1} (2)} \ldots \ldots u_{\sigma^{-1} (p + q)},\\
u \swarrow v &=& \sum_{\sigma \in \smop{Sh} (p, q),\, \sigma^{-1}(1)= 1 \smop{ and }\sigma^{-1}(p+q)= p+q} u_{\sigma^{-1} (1)}  u_{\sigma^{-1} (2)} \ldots \ldots u_{\sigma^{-1} (p + q)}.
\end{eqnarray}
Informally,
\begin{itemize}
\item $u\searrow v$ is the sum of words obtained by shuffling $u$ and $v$ in such a manner that the first letter is the first letter of $v$ and the last letter is the last letter of $v$,
\item $u\nearrow v$ is the sum of words obtained by shuffling $u$ and $v$ in such a manner that the first letter is the first letter of $v$ and the last letter is the last letter of $u$,
\item $u\nwarrow v$ is the sum of words obtained by shuffling $u$ and $v$ in such a manner that the first letter is the first letter of $u$ and the last letter is the last letter of $u$,
\item $u\swarrow v$ is the sum of words obtained by shuffling $u$ and $v$ in such a manner that the first letter is the first letter of $u$ and the last letter is the last letter of $v$.
\end{itemize}
Similar expressions for $\prec,\succ,\wedge,\vee$ and $\shu$ are straightforward and left to the reader.
\end{remark}
\noindent We can now state the main result of this article.
\begin{theorem}\label{thm:main}
Let $m$ be the concatenation product of words. The eight following diagrams commute.
\diagramme{
\xymatrix{\Cal H\otimes\Cal H^+\otimes\Cal H^+\ar[rr]^{I\otimes m} \ar[d]_{\tau_{23}\circ\Delta\otimes I\otimes I}&&
\Cal H\otimes\Cal H^+\ar[dd]^{\nearrow}\\
\Cal H\otimes\Cal H^+\otimes\Cal H\otimes\Cal H^+\ar[d]_{\searrow\otimes\wedge}&&\\
\Cal H^+\otimes\Cal H^+\ar[rr]_m &&\Cal H^+
}
\hskip 16mm
\xymatrix{\Cal H\otimes\Cal H^+\otimes\Cal H^+\ar[rr]^{I\otimes m} \ar[d]_{\tau_{23}\circ\Delta\otimes I\otimes I}&&
\Cal H\otimes\Cal H^+\ar[dd]^{\nearrow}\\
\Cal H\otimes\Cal H^+\otimes\Cal H\otimes\Cal H^+\ar[d]_{\succ\otimes\nearrow}&&\\
\Cal H^+\otimes\Cal H^+\ar[rr]_m &&\Cal H^+
}
}
\diagramme{
\xymatrix{\Cal H\otimes\Cal H^+\otimes\Cal H^+\ar[rr]^{I\otimes m} \ar[d]_{\tau_{23}\circ\Delta\otimes I\otimes I}&&
\Cal H\otimes\Cal H^+\ar[dd]^{\searrow}\\
\Cal H\otimes\Cal H^+\otimes\Cal H\otimes\Cal H^+\ar[d]_{\searrow\otimes\vee}&&\\
\Cal H^+\otimes\Cal H^+\ar[rr]_m &&\Cal H^+
}
\hskip 16mm
\xymatrix{\Cal H\otimes\Cal H^+\otimes\Cal H^+\ar[rr]^{I\otimes m} \ar[d]_{\tau_{23}\circ\Delta\otimes I\otimes I}&&
\Cal H\otimes\Cal H^+\ar[dd]^{\searrow}\\
\Cal H\otimes\Cal H^+\otimes\Cal H\otimes\Cal H^+\ar[d]_{\succ\otimes\searrow}&&\\
\Cal H^+\otimes\Cal H^+\ar[rr]_m &&\Cal H^+
}
}
\diagramme{
\xymatrix{\Cal H\otimes\Cal H^+\otimes\Cal H^+\ar[rr]^{I\otimes m} \ar[d]_{\tau_{23}\circ\Delta\otimes I\otimes I}&&
\Cal H\otimes\Cal H^+\ar[dd]^{\swarrow}\\
\Cal H\otimes\Cal H^+\otimes\Cal H\otimes\Cal H^+\ar[d]_{\swarrow\otimes\vee}&&\\
\Cal H^+\otimes\Cal H^+\ar[rr]_m &&\Cal H^+
}
\hskip 16mm
\xymatrix{\Cal H\otimes\Cal H^+\otimes\Cal H^+\ar[rr]^{I\otimes m} \ar[d]_{\tau_{23}\circ\Delta\otimes I\otimes I}&&
\Cal H\otimes\Cal H^+\ar[dd]^{\swarrow}\\
\Cal H\otimes\Cal H^+\otimes\Cal H\otimes\Cal H^+\ar[d]_{\prec\otimes\searrow}&&\\
\Cal H^+\otimes\Cal H^+\ar[rr]_m &&\Cal H^+
}
}
\diagramme{
\xymatrix{\Cal H\otimes\Cal H^+\otimes\Cal H^+\ar[rr]^{I\otimes m} \ar[d]_{\tau_{23}\circ\Delta\otimes I\otimes I}&&
\Cal H\otimes\Cal H^+\ar[dd]^{\nwarrow}\\
\Cal H\otimes\Cal H^+\otimes\Cal H\otimes\Cal H^+\ar[d]_{\swarrow\otimes\wedge}&&\\
\Cal H^+\otimes\Cal H^+\ar[rr]_m &&\Cal H^+
}
\hskip 16mm
\xymatrix{\Cal H\otimes\Cal H^+\otimes\Cal H^+\ar[rr]^{I\otimes m} \ar[d]_{\tau_{23}\circ\Delta\otimes I\otimes I}&&
\Cal H\otimes\Cal H^+\ar[dd]^{\nwarrow}\\
\Cal H\otimes\Cal H^+\otimes\Cal H\otimes\Cal H^+\ar[d]_{\prec\otimes\nearrow}&&\\
\Cal H^+\otimes\Cal H^+\ar[rr]_m &&\Cal H^+
}
}
In other words, for any $u\in\Cal H$ and $v, w \in \Cal H^+$ we have:
\begin{enumerate}
\item \begin{eqnarray*}
 u \nearrow (v w) &=& \sum_{u = u^1 u^2} (u^1 \searrow v) ( u^2 \wedge w) \\
&=& \sum_{u = u^1 u^2} ( u^1 \succ v)  ( u^2 \nearrow w), 
\end{eqnarray*}
\item \begin{eqnarray*}
 u \searrow (v w) &=& \sum_{u = u^1 u^2} (u^1 \searrow v) ( u^2 \vee w) \\
&=& \sum_{u = u^1 u^2} ( u^1 \succ v)  ( u^2 \searrow w), 
\end{eqnarray*}
\item \begin{eqnarray*}
 u \swarrow (v w) &=& \sum_{u = u^1 u^2} (u^1 \swarrow v) ( u^2 \vee w) \\
&=& \sum_{u = u^1 u^2} ( u^1 \prec v)  ( u^2 \searrow w), 
\end{eqnarray*}
\item \begin{eqnarray*}
 u \nwarrow (v w) &=& \sum_{u = u^1 u^2} (u^1 \swarrow v) ( u^2 \wedge w) \\
&=& \sum_{u = u^1 u^2} ( u^1 \prec v)  ( u^2 \nearrow w). 
\end{eqnarray*}
\end{enumerate}
\end{theorem}
\begin{proof}
We will prove this theorem by induction on the length of $u$. Let us verify that the theorem is true for $u = \un $ and for $u\in V$.\\

\noindent For $u = \un$ and for $v, w \in \Cal H^+$ we have:
\begin{eqnarray*}
 u \nearrow (v w) = \un \nearrow (v w) &=& 0,
\end{eqnarray*}
and: 
\begin{eqnarray*}
 \sum_{u = u^1 u^2} (u^1 \searrow v) ( u^2 \wedge w) &=& (\un \searrow v) \underbrace{(\un \wedge w)}_{0}\\
&=& 0\\
&=& \un \nearrow (v w),
\end{eqnarray*}
\begin{eqnarray*}
\sum_{u = u^1 u^2} ( u^1 \succ v)  ( u^2 \nearrow w) &=& (\un \succ v)  \underbrace{(\un \nearrow w)}_{0}\\
&=& 0\\
&=& \un \nearrow (v w).
\end{eqnarray*}
Similarly:
\begin{eqnarray*}
 u \searrow (v w) = \un \searrow vw &=& vw.
\end{eqnarray*}
\begin{eqnarray*}
\sum_{u = u^1 u^2} (u^1 \searrow v) ( u^2 \vee w) &=& (\un \searrow v) (\un \vee w)\\ 
&=& vw \\
&=& \un \searrow (v w).
\end{eqnarray*}
\begin{eqnarray*}
\sum_{u = u^1 u^2} ( u^1 \succ v)  ( u^2 \searrow w) &=& (\un \succ v) (\un\searrow w) \\
&=& vw \\
&=& \un \searrow (v w),
\end{eqnarray*}
and by a similar computation, we prove that the two other assertions are true for $u = \un$.\\ 

\noindent For $u \in V$ and for any $v, w \in \Cal H^+$, we have:
\begin{eqnarray*}
 u \nearrow (v w) &=& v w u,
\end{eqnarray*}
and: 
\begin{eqnarray*}
 \sum_{u = u^1 u^2} (u^1 \searrow v) ( u^2 \wedge w) &=& (\un \searrow v) ( u \wedge w) + (u \searrow v) \underbrace{(\un \wedge w)}_{0}\\
&=& vwu\\
&=& u \nearrow (v w),
\end{eqnarray*}
\begin{eqnarray*}
\sum_{u = u^1 u^2} ( u^1 \succ v)  ( u^2 \nearrow w) &=& (\un \succ v)  ( u \nearrow w) + (u \succ v)  \underbrace{(\un \nearrow w)}_{0}\\ 
&=& v w u\\
&=& u \nearrow (v w).
\end{eqnarray*}
Similarly:
\begin{eqnarray*}
 u \searrow (v w) &=& vuw.
\end{eqnarray*}
\begin{eqnarray*}
\sum_{u = u^1 u^2} (u^1 \searrow v) ( u^2 \vee w) &=& (\un \searrow v) ( u \vee w) + \underbrace{ (u \searrow v)}_{0} (\un \vee w)\\
&=& vuw \\
&=& u \searrow (v w).
\end{eqnarray*}
\begin{eqnarray*}
\sum_{u = u^1 u^2} ( u^1 \succ v)  ( u^2 \searrow w) &=& (\un \succ v) \underbrace{( u \searrow w)}_{0} + ( u \succ v)  (\un\searrow w) \\
&=& vuw \\
&=& u \searrow (v w).
\end{eqnarray*}

\noindent By a similar computation, we prove that the two other assertions are true for $u\in V$. \\

\noindent We will now use the induction hypothesis to prove the theorem. Let $u = a\theta b$, $v$ and $w$ be three elements of $\Cal H^+$.\\

\noindent \textbf{Proof of (1):}
\begin{eqnarray*}
\sum_{u = u^1 u^2} (u^1 \searrow v) ( u^2 \wedge w)  &=& \underbrace{(u \searrow v)(\un \wedge w)}_{0} + \sum_{{u = u^1 u^2}\atop u^2 \neq \mop{\tiny{\un}}} ( u^1 \searrow v)  (u^2 \wedge w).
\end{eqnarray*}
The condition $u^2 \neq \un$ gives: $u^2 = u^{12} b$ where $u^{1} u^{12} = a \theta$, hence:   
\begin{eqnarray*}
\sum_{u = u^1 u^2} ( u^1 \searrow v)  ( u^2 \wedge w) &=& \sum_{a \theta = u^{1} u^{12}} (u^{1} \searrow  v)  (u^{12} b \wedge  w).
\end{eqnarray*}

We distinguish here two cases, the first case where $v$ is a single-letter word and the second case where $v$ is a word of length $\geq 2$, i.e $v = c \xi d$, where $c, d \in V$ and $\xi \in V^{\otimes n}$.\\

If $v \in  V$, by Remark \ref{comm} we have $u^{1} \searrow  v = v \nwarrow u^1 = 0$ for all $u^{1} \neq \un$ (see equation \eqref{nul}). Hence the sum $\sum_{a \theta = u^{1} u^{12}} (u^{1} \searrow  v)  (u^{12} b \wedge  w)$ gives one term where $u^{1} = \un$, the other terms all vanish, and thus:

\begin{eqnarray*}
\sum_{u = u^1 u^2} ( u^1 \searrow v)  ( u^2 \wedge w) &=&\sum_{a \theta = u^{1} u^{12}} (u^{1} \searrow  v)  (u^{12} b \wedge  w)\\
&=& (\un \searrow v)  ( u \wedge w)\\
&=& v( u \wedge w)\\
&=& u \nearrow ( v w).
\end{eqnarray*}
\noindent The last equality is valid because $v$ is a single letter here.\\

\noindent Now if $v = c \xi d$ we obtain the same result:
\begin{eqnarray*}
\sum_{u = u^1 u^2} ( u^1 \searrow v)  ( u^2 \wedge w) &=&  \sum_{a \theta = u^{1} u^{12}} (u^{1} \searrow  v)  (u^{12} b \wedge  w)\\
&=&  \sum_{a \theta = u^{1} u^{12}}  (u^{1} \succ c \xi) d (u^{12} \shu w) b\\
&=&  \sum_{a \theta = u^{1} u^{12}}  (u^{1} \succ c \xi) (u^{12} \succ dw) b\\
&=& \sum_{a \theta = u^{1} u^{12}} \big[ (u^{1} \succ c \xi) (u^{12} \nearrow d w) b + (u^{1} \succ c \xi) (u^{12} \searrow d w) b \big]\\
&=& (a \theta\nearrow v w) b + (a \theta \searrow v w) b \;\;\;\;\text{(induction hypothesis)}\\
&=& (a \theta \succ v w) b\\
&=& (a \theta b) \nearrow ( c\xi d w)\\
&=& u \nearrow ( v w).
\end{eqnarray*}

\noindent Similary, we have:
\begin{eqnarray*}
\sum_{u = u^1 u^2} ( u^1 \succ v)  ( u^2 \nearrow w)  &=& \underbrace{(u \succ v)(\un \nearrow w)}_{0} + \sum_{{u = u^1 u^2}\atop {u^1 \neq \; u , u^2 \neq \mop{\tiny{\un}}}} ( u^1 \succ v)  ( u^2 \nearrow w).
\end{eqnarray*}
The condition $u^1 \neq \; u$ and $u^2 \neq \un$ gives: $u^2 = u^{12} b$ where $u^{1}u^{12} = a \theta$, hence:   
\begin{eqnarray*}
\sum_{u = u^1 u^2} ( u^1 \succ v)  ( u^2 \nearrow w) &=& \sum_{a \theta = u^{1} u^{12}} (u^{1} \succ v)  (u^{12} b \nearrow w)\\ 
&=& \sum_{a \theta = u^{1} u^{12}} (u^{1} \succ v)  (u^{12} \succ w) b\\
&=& \sum_{a \theta = u^{1} u^{12}} (u^{1} \succ v)  (u^{12} \nearrow w) b + \sum_{a \theta = u^{1} u^{12}} (u^{1} \succ v)  (u^{12} \searrow w) b \\
&=& (a \theta \swarrow v w) b + (a \theta \nwarrow v w) b  \;\;\;\;\text{(induction hypothesis)}\\
&=& (a \theta \succ v w) b\\
&=& c (a \theta \shu \xi d w) b\\
&=& (a \theta b) \nearrow ( c\xi d w)\\
&=& u \nearrow ( v w).
\end{eqnarray*}
\noindent \textbf{Proof of (2):} By a similar method we prove the second assertion. We distinguish here two cases, the first case where $v$ is a single-letter word and the second case where $v$ is a word of length $\geq 2$, i.e $v = c \xi d$, where $c, d \in V$ and $\xi \in V^{\otimes n}$.\\
 
\noindent If $v \in  V$, by Remark \ref{comm} we have $u^{1} \searrow  v = v \nwarrow u^1 = 0$ for all $u^{1} \neq \un$ (see equation \eqref{nul}). Hence the sum $\sum_{u = u^1 u^2} ( u^1 \searrow v)  ( u^2 \vee w)$ gives one term where $u^1 = \un$ and $u^2 = u$, the other terms all vanish, we have:

\begin{eqnarray*}
\sum_{u = u^1 u^2} ( u^1 \searrow v)  ( u^2 \vee w) &=& (\un \searrow v) (u  \vee w)\\  
&=& v (u  \vee w)\\  
&=& u \searrow (v w),
\end{eqnarray*}
\noindent and if $v = c \xi d$, we have:
\begin{eqnarray*}
\sum_{u = u^1 u^2} ( u^1 \searrow v)  ( u^2 \vee w) &=& \sum_{u = u^{1} u^{2}} (u^{1} \searrow c \xi d)  (u^{2}  \vee w)\\  
&=& \sum_{u = u^{1} u^{2}} \sum_{u^1 = u^{11} u^{12}}(u^{11} \searrow c ) (u^{12}  \vee \xi d) (u^{2}  \vee w)\;\;\text{(induction hypothesis)}\\
&=& \sum_{u = u^{11} u^{12} u^{2}} (u^{11} \searrow c ) (u^{12}  \vee \xi d) (u^{2}  \vee w)\\
&=& \sum_{u = u^{11} u'} (u^{11} \searrow c ) \sum_{u' = u^{11} u^{12}}(u^{12}  \vee \xi d) (u^{2}  \vee w)\\
&=& \sum_{u = u^{11} u'} (u^{11} \searrow c ) (u'  \vee \xi dw) \;\;\text{(induction hypothesis)}\\
&=& \sum_{u = u^{11} u'} (c \nwarrow u^{11}) (u'  \vee \xi d w).\\
\end{eqnarray*}
The last sum contains one term because $c \nwarrow u^{11} = 0$ if $u^{11} \neq \un$, then we have:
\begin{eqnarray*}
\sum_{u = u^1 u^2} ( u^1 \searrow v)  ( u^2 \vee w) &=& c (u  \vee \xi d w)\\
%&=& c (u \vee \xi d w) \\
&=& u \searrow (c\xi d w)\\
&=& u \searrow (v w).
\end{eqnarray*}

\noindent Similarly, we distinguish here two cases, the first case where $w$ is a single-letter word and the second case where $w$ is a word of length $\geq 2$, i.e $w = e \eta f$, where $e, f \in V$ and $\eta \in V^{\otimes n}$.\\

\noindent If $w \in  V$, the sum $\sum_{u = u^1 u^2} ( u^1 \succ v) ( u^2 \searrow w)$ gives one term where $u^1 = u$ and $u^2 = \un$, the other terms vanish, which gives:
\begin{eqnarray*}
\sum_{u = u^1 u^2} ( u^1 \succ v)  ( u^2 \searrow w) &=& (u \succ v)  (\un  \searrow w)\\ 
&=& (u \succ v) w\\
&=& u \searrow (v w),
\end{eqnarray*}
and if  $w = e \eta f$, we have:
 \begin{eqnarray*}
\sum_{u = u^1 u^2} ( u^1 \succ v)  ( u^2 \searrow w) &=& \sum_{u = u^{1} u^{2}} (u^{1} \succ v)  (u^{2}  \searrow e \eta f)\\ 
&=& \sum_{u = u^{1} u^{2}} \sum_{u^2 = u^{21} u^{22}} (u^{1} \succ v)  (u^{21}  \succ e \eta) (u^{22} \searrow f)\;\;\text{(induction hypothesis)}\\
&=& \sum_{u = u^{1} u^{21} u^{22}} (u^{1} \succ v)  (u^{21}  \succ e \eta) (u^{22}\searrow f)\\
&=& \sum_{u = u' u^{22}} \sum_{u' = u^{1} u^{21}}(u^{1} \succ v)  (u^{21}  \succ e \eta) (u^{22} \searrow f)\\
&=& \sum_{u = u' u^{22}} (u' \succ v e \eta) (u^{22} \searrow f)\;\;\text{(induction hypothesis)}\\
&=& \sum_{u = u' u^{22}} (u' \succ v e \eta) (f\nwarrow u^{22}).
\end{eqnarray*}
\noindent The last sum contains one term because $f \nwarrow u^{22} = 0$ if $u^{22} \neq \un$, then we obtain:
\begin{eqnarray*}
\sum_{u = u^1 u^2} ( u^1 \succ v)  ( u^2 \searrow w) &=&  (u \succ v e \eta) f\\ 
&=&  (u \succ c \xi d e \eta) f\\
&=&  c (u \shu \xi d e \eta) f\\
&=& u \searrow (c\xi d e \eta f)\\
&=& u \searrow (v w).
\end{eqnarray*}
\noindent \textbf{Proof of (3):}
\begin{eqnarray*}
\sum_{u = u^1 u^2} ( u^1 \swarrow v)  ( u^2 \vee w)  &=& \underbrace{(\un \swarrow v)(u \vee w)}_{0} + \sum_{{u = u^1 u^2}\atop {u^1 \neq \mop{\tiny{\un}}, u^2 \neq u}} ( u^1 \swarrow v)  ( u^2 \vee w).
\end{eqnarray*}
\noindent The condition $u^1 \neq \un$ gives: $u^1 = a u^{11}$ where $u^{11}u^{2} = \theta b$, hence:   
\begin{eqnarray*}
\sum_{u = u^1 u^2}  ( u^1 \swarrow v)  ( u^2 \vee w)  &=& \sum_{u = a u^{11} u^2}  (a u^{11}  \swarrow v)  ( u^2 \vee w)\\ 
&=& \sum_{u = a u^{11} u^2}   a (u^{11} \vee v)  ( u^2 \vee w) \\
&=& \sum_{u = a u^{11} u^2}   a (u^{11} \swarrow  v) ( u^2 \vee w) +  \sum_{u = a u^{11} u^2}   a (u^{11} \searrow v )  ( u^2 \vee w) \\
&=& a (\theta b \swarrow  v w) + a (\theta b \searrow v w)   \;\;\;\;\text{(induction hypothesis)}\\
&=& a (\theta b \vee  v w) \\
&=& (a \theta b) \swarrow (vw)\\
&=& u \swarrow ( v w).
\end{eqnarray*}

\noindent Similarly, we have:
 \begin{eqnarray*}
\sum_{u = u^1 u^2} ( u^1 \prec v)  ( u^2 \searrow w)  &=& \underbrace{(\un \prec v)(u \searrow w)}_{0} + \sum_{{u = u^1 u^2}\atop {u^1 \neq \mop{\tiny{\un}}, u^2 \neq u}} ( u^1 \prec v)  ( u^2 \searrow w).
\end{eqnarray*}
\noindent The condition $u^1 \neq \un$ gives: $u^1 = a u^{11}$ where $u^{11}u^{2} = \theta b$, hence: 
\begin{eqnarray*}
\sum_{u = u^1 u^2}  ( u^1 \prec v)  ( u^2 \searrow w)  &=& \sum_{u = a u^{11} u^2}  (a u^{11} \prec  v)  (u^2 \searrow w).
\end{eqnarray*}
\noindent We distinguish here two cases, the first case where $w$ is a single-letter word and the second case where $w$ is a word of length $\geq 2$, i.e $w = e \eta f$, where $e, f \in V$ and $\eta \in V^{\otimes n}$.\\

\noindent If $w \in  V$, the sum $\sum_{u = a u^{11} u^2}  (a u^{11} \prec  v)  (u^2 \searrow w)$ gives one term where $u^2 = \un$, the other terms  all vanish, we have: 
\begin{eqnarray*}
\sum_{u = u^1 u^2}  (u^1 \prec v)  ( u^2 \searrow w) &=&\sum_{u = a u^{11} u^2}  (a u^{11} \prec  v)  (u^2 \searrow w)\\ &=&  (u \prec v)  (\un \searrow w)\\
&=&  (u \prec v) w \\  
&=& u \swarrow ( v w).
\end{eqnarray*}	
\noindent Now if $w = e \eta f$, we have:
\begin{eqnarray*}
\sum_{u = u^1 u^2}  ( u^1 \prec v)  ( u^2 \searrow w) &=& \sum_{u = a u^{11} u^2}  (a u^{11} \prec  v)  (u^2 \searrow w)\\ 
&=& \sum_{u = a u^{11} u^2}  (a u^{11} \prec  v)  (u^2 \searrow e \eta f)\\ 
&=& \sum_{u = a u^{11} u^2} a (u^{11} \shu  v)  e(u^2 \shu \eta) f\\ 
&=& \sum_{u = a u^{11} u^2} a (u^{11} \vee  v e) (u^2 \vee \eta f) \\ 
&=& \sum_{u = a u^{11} u^2} a (u^{11} \searrow  v e) (u^2 \vee \eta f) + \sum_{u = a u^{11} u^2} a (u^{11} \swarrow  v e) (u^2 \vee \eta f) \\ 
&=& a (\theta b \searrow  v e \eta f) + a (\theta b \swarrow  v e \eta f)  \;\;\;\;\text{(induction hypothesis)}\\
 &=& a (\theta b \vee  v w)  \\
&=& u \swarrow ( v w),
\end{eqnarray*}
which proves the third assertion.\\

\noindent \textbf{Proof of (4):}
\begin{eqnarray*}
\sum_{u = u^1 u^2} ( u^1  \swarrow v)  ( u^2 \wedge w)  &=& \underbrace{(\un \swarrow v)(u \wedge w) + (u \swarrow v)(\un \wedge w)}_{0} + \sum_{{u = u^1 u^2}\atop {u^1 , u^2 \neq \mop{\tiny{\un}}, u}} ( u^1 \swarrow v)  ( u^2 \wedge w).
\end{eqnarray*}
\noindent The condition $u^1 , u^2 \neq \un, u$ gives: $u^1 = a u^{11}$ and $u^2 = u^{12} b$ where $u^{11}u^{12} = \theta$, hence:   
\begin{eqnarray*}
\sum_{u = u^1 u^2} ( u^1 \swarrow v)  ( u^2 \wedge w) &=& \sum_{\theta = u^{11} u^{12}} (a u^{11} \swarrow v)  (u^{12} b \shu w) b\\ 
&=& \sum_{\theta = u^{11} u^{12}} (a u^{11} \swarrow v)  (u^{12}  \vee w) b + \sum_{\theta = u^{11} u^{12}} (a u^{11} \swarrow v)  (u^{12}  \wedge w) b \\
&=& (a \theta \swarrow  v w) b + (a \theta  \nwarrow  v w) b  \;\;\;\;\text{(induction hypothesis)}\\
&=& (a \theta \prec v w) b \\
&=& a (\theta \shu v w) b\\
&=& (a \theta b) \nwarrow (v w)\\
&=& u \nwarrow (v w).
\end{eqnarray*}
\noindent Similarly we have:
\begin{eqnarray*}
\sum_{u = u^1 u^2} ( u^1 \prec v)  ( u^2 \nearrow w) &=& \underbrace{(\un \prec v)(u \nearrow w) + (u \prec v)(\un \nearrow w)}_{0} + \sum_{{u = u^1 u^2}\atop {u^1 , u^2 \neq \mop{\tiny{\un}}, u}} ( u^1 \prec v)  ( u^2 \nearrow w).
\end{eqnarray*}
\noindent The condition $u^1 , u^2 \neq \un, u$ gives: $u^1 = a u^{11}$ and $u^2 = u^{12} b$ where $u^{11}u^{12} = \theta$, hence:
\begin{eqnarray*}
\sum_{u = u^1 u^2} ( u^1 \prec v)  ( u^2 \nearrow w) &=& \sum_{\theta = u^{11} u^{12}} (a u^{11} \prec v)  (u^{12} b \nearrow w)\\
&=& \sum_{\theta = u^{11} u^{12}} (a u^{11} \prec v)  (u^{12} \succ w) b\\
&=& \sum_{\theta = u^{11} u^{12}} (a u^{11} \prec v)  (u^{12} \nearrow w) b + \sum_{\theta = u^{11} u^{12}} (a u^{11} \prec v)  (u^{12} \searrow w) b\\
&=& (a \theta \swarrow v w ) b + (a \theta \nwarrow v w ) b  \;\;\;\;\text{(induction hypothesis)}\\
&=& (a \theta \prec v w ) b\\
&=& a (\theta \shu v w ) b\\
&=& (a \theta b) \nwarrow (v w ) \\
&=& u \nwarrow (v w ),
\end{eqnarray*}
which proves the fourth assertion.
\end{proof}
\begin{remark}\rm
A non-recursive proof of Theorem \ref{thm:main} is available. Indeed, to prove the first assertion of (2) we note that $u\searrow (vw)$ is obtained by summing all terms in the shuffle of $u$ with $vw$ so that the first letter belongs to $v$ and the last letter belongs to $w$. We cut each of these terms just after the last letter of $v$. The left part is obtained by shuffling a prefix of $u$ with $v$ such that the first and last letters are in $v$. The right part is obtained by shuffling a suffix of $u$ with $w$ such that the last letter is in $w$. We proceed similarly for the second assertion, cutting just before the first letter of $w$. Items (1), (3) and (4) can be handled similarly.
\end{remark}

\begin{corollary}\label{cor-one}
The four following diagrams commute:
\diagramme{
\xymatrix{\Cal H\otimes\Cal H^+\otimes\Cal H^+\ar[rr]^{I\otimes m} \ar[d]_{\tau_{23}\circ\Delta\otimes I\otimes I}&&
\Cal H\otimes\Cal H^+\ar[dd]^{\wedge}\\
\Cal H\otimes\Cal H^+\otimes\Cal H\otimes\Cal H^+\ar[d]_{\vee\otimes\wedge}&&\\
\Cal H^+\otimes\Cal H^+\ar[rr]_m &&\Cal H^+
}
\hskip 16mm
\xymatrix{\Cal H\otimes\Cal H^+\otimes\Cal H^+\ar[rr]^{I\otimes m} \ar[d]_{\tau_{23}\circ\Delta\otimes I\otimes I}&&
\Cal H\otimes\Cal H^+\ar[dd]^{\prec}\\
\Cal H\otimes\Cal H^+\otimes\Cal H\otimes\Cal H^+\ar[d]_{\prec\otimes\succ}&&\\
\Cal H^+\otimes\Cal H^+\ar[rr]_m &&\Cal H^+
}
}
\diagramme{
\xymatrix{\Cal H\otimes\Cal H^+\otimes\Cal H^+\ar[rr]^{I\otimes m} \ar[d]_{\tau_{23}\circ\Delta\otimes I\otimes I}&&
\Cal H\otimes\Cal H^+\ar[dd]^{\vee}\\
\Cal H\otimes\Cal H^+\otimes\Cal H\otimes\Cal H^+\ar[d]_{\vee\otimes\vee}&&\\
\Cal H^+\otimes\Cal H^+\ar[rr]_m &&\Cal H^+
}
\hskip 16mm
\xymatrix{\Cal H\otimes\Cal H^+\otimes\Cal H^+\ar[rr]^{I\otimes m} \ar[d]_{\tau_{23}\circ\Delta\otimes I\otimes I}&&
\Cal H\otimes\Cal H^+\ar[dd]^{\succ}\\
\Cal H\otimes\Cal H^+\otimes\Cal H\otimes\Cal H^+\ar[d]_{\succ\otimes\succ}&&\\
\Cal H^+\otimes\Cal H^+\ar[rr]_m &&\Cal H^+
}
}
In other words, given $u\in\Cal H$ and $v,w\in\Cal H^+$ we have:
\begin{enumerate}
\item \begin{eqnarray*}
 u  \wedge ( v w) &=& u \nearrow (v w) +  u \nwarrow (v w)\\
 &=& \sum_{u = u^1 u^2} (u^1 \vee v) ( u^2 \wedge w), 
\end{eqnarray*}
\item \begin{eqnarray*}
u  \prec ( v w)  &=& u \nwarrow (v w) + u \swarrow (v w)\\
 &=& \sum_{u = u^1 u^2} (u^1 \prec v) ( u^2 \succ w),
\end{eqnarray*}
\item \begin{eqnarray*}
u  \vee ( v w) &=& u \swarrow (v w) + u \searrow (v w)  \\
&=& \sum_{u = u^1 u^2} (u^1 \vee v) ( u^2 \vee w),
\end{eqnarray*}
\item \begin{eqnarray*}
 u  \succ ( v w)  &=& u \searrow (v w) +  u \nearrow (v w) \\
&=& \sum_{u = u^1 u^2} (u^1 \succ v) ( u^2 \succ w).
\end{eqnarray*}
\end{enumerate}
\end{corollary}
\begin{proof} 
The diagram in position $(1,1)$ of the $2\times 2$-matrix of diagrams above is obtained by adding both diagrams $(4,1)$ and $(1,1)$ in the $4\times 2$-matrix of diagrams of Theorem \ref{thm:main}. Similarly, diagram $(1,2)$ is obtained by adding both diagrams $(3,2)$ and $(4,2)$ of Theorem \ref{thm:main}, diagram $(2,1)$ is obtained by adding $(2,1)$ and $(3,1)$ of Theorem \ref{thm:main}, and finally $(2,2)$ is obtained as the sum of $(1,2)$ and $(2,2)$ thereof.
\end{proof}
\begin{corollary}\label{theo3} Both following diagrams commute:
\diagramme{
\xymatrix{\Cal H\otimes\Cal H^+\otimes\Cal H^+\ar[rr]^{I\otimes m} \ar[d]_{\tau_{23}\circ\Delta\otimes I\otimes I}&&
\Cal H\otimes\Cal H^+\ar[dd]^{\sshu}\\
\Cal H\otimes\Cal H^+\otimes\Cal H\otimes\Cal H^+\ar[d]_{\vee\otimes\sshu}&&\\
\Cal H^+\otimes\Cal H^+\ar[rr]_m &&\Cal H^+
}
\hskip 16mm
\xymatrix{\Cal H\otimes\Cal H^+\otimes\Cal H^+\ar[rr]^{I\otimes m} \ar[d]_{\tau_{23}\circ\Delta\otimes I\otimes I}&&
\Cal H\otimes\Cal H^+\ar[dd]^{\sshu}\\
\Cal H\otimes\Cal H^+\otimes\Cal H\otimes\Cal H^+\ar[d]_{\sshu\otimes\succ}&&\\
\Cal H^+\otimes\Cal H^+\ar[rr]_m &&\Cal H^+
}
}
In other words, for any $u\in\Cal H$ and $v, w \in \Cal H^+$, we have:
\begin{eqnarray}
 u \shu (v w) &=& \sum_{u = u^1 u^2} (u^1 \vee v) ( u^2 \shu w) \\
&=& \sum_{u = u^1 u^2} ( u^1 \shu v)  ( u^2 \succ w). 
\end{eqnarray}
\end{corollary}

\begin{proof}
The first diagram is obtained by adding diagrams $(1.1)$ and $(2.1)$ of Corollary \ref{cor-one}, the second is obtained by adding $(1.2)$ and $(2.2)$ thereof.
\end{proof}
\begin{ex}
An example of computation for $u = u_1 u_2 \in V^{\otimes 2}$,  $v = v_1 v_2 \in V^{\otimes 2}$ and $w \in V$.
\begin{eqnarray*}
 u \shu (v w) &=& (u_1 u_2) \shu (v_1 v_2 w) \\
 &=& u_1 u_2 v_1 v_2 w + u_1 v_1 u_2 v_2 w  + u_1 v_1 v_2 u_2 w + u_1 u_2 v_1 v_2 w  + u_1 v_1 v_2 w u_2 \\
&\hskip 10mm +& v_1 v_2 u_1 u_2 w + v_1 u_1 v_2 u_2 w + v_1 u_1 v_2 w u_2  + v_1 v_2 u_1 w u_2 + v_1 v_2 w u_1 u_2.
\end{eqnarray*}
Also we have:
\begin{eqnarray*}
 \sum_{u = u^1 u^2} (u^1 \vee v) ( u^2 \shu w) &=& (\un \vee v) (u \shu w) + (u_1 \vee v) (u_2 \shu w) + (u_1u_2 \vee v) (\un \shu w)  \\
&=&  v (u_1  u_2 w + u_1 w u_2 + w u_1 u_2) + (u_1 v_1 v_2 + v_1 u_1 v_2) (u_2  w + w u_2)\\
&\hskip 10mm +& (u_1 u_2 v_1 v_2 + u_1 v_1 u_2 v_2 + v_1 u_1 u_2 v_2 ) w  \\
&=&  v_1 v_2 u_1  u_2 w + v_1 v_2 u_1 w u_2 + v_1 v_2 w u_1 u_2 + u_1 v_1 v_2 u_2  w + v_1 u_1 v_2 u_2  w \\
&\hskip 10mm +& u_1 v_1 v_2 w u_2 + v_1 u_1 v_2 w u_2 + u_1 u_2 v_1 v_2 w + u_1 v_1 u_2 v_2 w + v_1 u_1 u_2 v_2 w, 
\end{eqnarray*}
and,
\begin{eqnarray*}
\sum_{u = u^1 u^2} ( u^1 \shu v)  ( u^2 \succ w)  &=& (\un \shu v) (u \succ w) + (u_1 \shu v) (u_2 \succ w) + (u_1u_2 \shu v) (\un \succ w)  \\
&=&  v_1 v_2 w u_1  u_2 + (u_1 v_1 v_2 + v_1 u_1 v_2 + v_1 v_2 u_1) w u_2 \\
&\hskip 10mm +& (u_1 u_2 v_1 v_2 + u_1 v_1 u_2 v_2 + v_1 u_1 u_2 v_2 + u_1 v_1 v_2 u_2 + v_1 u_1 v_2 u_2 + v_1 v_2 u_2 u_2) w  \\
&=& v_1 v_2 w u_1  u_2 + u_1 v_1 v_2  w u_2 + v_1 u_1 v_2 w u_2 + v_1 v_2 u_1 w u_2 + u_1 u_2 v_1 v_2 w \\
&\hskip 10mm +& u_1 v_1 u_2 v_2 w + v_1 u_1 u_2 v_2 w + u_1 v_1 v_2 u_2 w + v_1 u_1 v_2 u_2 w + v_1 v_2 u_2 u_2 w.
\end{eqnarray*}
Then we have:
\begin{eqnarray*}
 u \shu (v w) &=& \sum_{u = u^1 u^2} (u^1 \vee v) ( u^2 \shu w) \\
&=&\sum_{u = u^1 u^2} ( u^1 \shu v)  ( u^2 \succ w).
\end{eqnarray*}
\end{ex}

%%%%%%%%%%%%%%%%%%%%%%%%%%%%%%%%%%%%%%%%%%%%%%%%%%%%%%%%%%%%%%%%%%%
%%%%%%%%%%%%%%%%%%%%%%%%%%%%%%%%%%%%%%%%%%%%%%%%%%%%%%%%%%%%%%%%%%%
%%%%%%%%%%%%%%%%%%%%%%%%%%%%%%%%%%%%%%%%%%%%%%%%%%%%%%%%%%%%%%%%%%%
\section{Module-algebra structures on the shuffle quadri-algebra}
%%%%%%%%%%%%%%%%%%%%%%%%%%%%%%%%%%%%%%%%%%%%%%%%%%%%%%%%%%%%%%%%%%%%%%%%%%%%%%%
%%%%%%%%%%%%%%%%%%%%%%%%%%%%%%%%%%%%%%%%%%%%%%%%%%%%%%%%%%%%%%%%%%%%%%%%%%%%%%%
%%%%%%%%%%%%%%%%%%%%%%%%%%%%%%%%%%%%%%%%%%%%%%%%%%%%%%%%%%%%%%%%%%%
\noindent We consider the bialgebra $(\Cal H, \shu, \Delta)$ and the non-unitary infinitesimal bialgebra $(\Cal H^+, m, \Delta)$. The infinitesimal bialgebra compatibility relation is written as:
\begin{equation}
\Delta(uv) = (\Delta u)(\un\otimes v) - u\otimes v + (u\otimes\un)(\Delta v).
\end{equation}
Here we consider the restriction to $\Cal H^+$ of the \textsl{full} deconcatenation coproduct
$$\Delta:\Cal H^+\to\Cal H^+\ootimes\Cal H^+.$$ 
\begin{proposition} Both maps  $\vee$ and $\succ$ are left actions of $(\Cal H, \shu)$ on $\Cal H^+$, and both maps  $\wedge$ and $\prec$ are right actions of $(\Cal H, \shu)$ on $\Cal H^+$. In other words, the four following diagrams are commutative: 
\diagramme{
\xymatrix{\Cal H^+ \ootimes \Cal H^+ \ootimes \Cal H^+ \ar[rrr]^{\sshu \ootimes I}\ar[d]_{I \ootimes \vee}
&&&\Cal H^+ \ootimes \Cal H^+ \ar[d]^{\vee}\\
\Cal H^+ \ootimes \Cal H^+
\ar[rrr]_{\vee }&&&\Cal H^+
}
\hskip 12mm
\xymatrix{\Cal H^+ \ootimes \Cal H^+ \ootimes \Cal H^+ \ar[rrr]^{I\ootimes\sshu}\ar[d]_{\wedge \ootimes I}
&&&\Cal H^+ \ootimes \Cal H^+ \ar[d]^{\wedge}\\
\Cal H^+ \ootimes \Cal H^+
\ar[rrr]_{\wedge }&&&\Cal H^+
}
} 
\diagramme{
\xymatrix{\Cal H^+ \ootimes \Cal H^+ \ootimes \Cal H^+ \ar[rrr]^{I \ootimes \succ}\ar[d]_{\sshu \ootimes I}
&&&\Cal H^+ \ootimes \Cal H^+ \ar[d]^{\succ}\\
\Cal H^+ \ootimes \Cal H^+
\ar[rrr]_{\succ }&&& \Cal H^+
}
\hskip 12mm
\xymatrix{\Cal H^+ \ootimes \Cal H^+ \ootimes \Cal H^+ \ar[rrr]^{\prec\ootimes I}\ar[d]_{I\ootimes \sshu}
&&&\Cal H^+ \ootimes \Cal H^+ \ar[d]^{\prec}\\
\Cal H^+ \ootimes \Cal H^+
\ar[rrr]_{\prec }&&& \Cal H^+
}
} 
\noindent That is to say: 
\begin{equation}
\vee \circ (I \ootimes \vee) = \vee \circ (\shu \ootimes I),\hskip 12mm \wedge \circ (\wedge \ootimes I) = \wedge \circ (I\ootimes\shu),\label{th4}
\end{equation}
\begin{equation}
\succ \circ (I \ootimes \succ) = \;\succ \circ (\shu \ootimes I), \hskip 12mm \prec \circ (\prec\ootimes I) = \prec \circ (I\ootimes\shu).\label{th5}
\end{equation}
\end{proposition}
\begin{proof}
This is immediate from the dendriform axioms.
\end{proof}

\begin{theorem} The dendriform products $\vee$ and $\succ$ define two $(\Cal H,m)$-module-algebra structures on $\Cal H^+$, i.e. the two following diagrams are commutative: 
\diagramme{
\xymatrix{
\Cal H \otimes \Cal H^+ \otimes \Cal H^+ \ar[d]_{I\otimes\; m} \ar[rr]^{\Delta  \otimes I \otimes I} 
&&\Cal H \otimes\Cal H \otimes \Cal H^+ \otimes \Cal H^+ \ar[d]^{\tau_{23} }\\
\Cal H \otimes\Cal H^+ \ar[d]_{\vee}&&\Cal H\otimes\Cal H^+ \otimes \Cal H\otimes\Cal H^+\ar[d]^{\vee \;\otimes\;\; \vee }\\
\Cal H^+
&& \ar[ll]^{m}\Cal H^+\otimes \Cal H^+ }
}

\diagramme{
\xymatrix{
\Cal H \otimes \Cal H^+ \otimes \Cal H^+ \ar[d]_{I\otimes\; m} \ar[rr]^{\Delta  \otimes I \otimes I} 
&&\Cal H \otimes\Cal H \otimes \Cal H^+ \otimes \Cal H^+ \ar[d]^{\tau_{23} }\\
\Cal H \otimes\Cal H^+ \ar[d]_{\succ}&&\Cal H\otimes\Cal H^+ \otimes \Cal H\otimes\Cal H^+\ar[d]^{\succ \;\otimes\;\; \succ }\\
\Cal H^+
&& \ar[ll]^{m}\Cal H^+\otimes \Cal H^+ }
}
\noindent That is to say: 
\begin{equation}
m \circ (\vee \;\otimes\;\; \vee)\circ \tau_{23} \circ (\Delta  \otimes I \otimes I) = \; \vee \circ (I\otimes\; m),\label{th3}
\end{equation}
\begin{equation}
m \circ (\succ \;\otimes\;\; \succ)\circ \tau_{23} \circ (\Delta  \otimes I \otimes I) = \; \succ \circ (I\otimes\; m).\label{th3}
\end{equation}

\end{theorem}

\begin{proof}
For more details on module-algebras, see e.g. \cite[Definition 4.1.1]{M93}. To prove the commutativity of these diagrams, we will use the results of Corollary \ref{cor-one}. For $u\in \Cal H$ and $v,w\in\Cal H^+$ we have:
\begin{eqnarray*}
m \circ (\vee \;\otimes\;\; \vee)\circ \tau_{23} \circ (\Delta  \otimes I \otimes I)(u \otimes v \otimes w) &=& m \circ (\vee \;\otimes\;\; \vee)\circ \tau_{23} \left(\sum_{u = u^1 u^2} u^1 \otimes u^2 \otimes v \otimes w\right)\\
 &=& m \circ (\vee \;\otimes\;\; \vee)\left(\sum_{u = u^1 u^2} u^1 \otimes v \otimes u^2 \otimes w\right)\\
&=& \sum_{u = u^1 u^2}( u^1 \vee v)( u^2 \vee w),
\end{eqnarray*}
\noindent whereas
\begin{eqnarray*}
\vee \circ (I\otimes\; m) (u \otimes v \otimes w) &=& \vee (u \otimes v w)\\
&=& u \vee (v  w)\\
&=& \sum_{u = u^1 u^2}( u^1 \vee v) ( u^2 \vee w).
\end{eqnarray*}

\noindent We also have:
\begin{eqnarray*}
m \circ (\succ \;\otimes\;\; \succ)\circ \tau_{23} \circ (\Delta  \otimes I \otimes I)(u \otimes v \otimes w) &=& m \circ (\succ \;\otimes\;\; \succ)\circ \tau_{23} \left(\sum_{u = u^1 u^2} u^1 \otimes u^2 \otimes v \otimes w\right)\\
 &=& m \circ (\succ \;\otimes\;\; \succ)\circ \left(\sum_{u = u^1 u^2} u^1 \otimes v \otimes u^2 \otimes w\right)\\
&=& \sum_{u = u^1 u^2}( u^1 \succ v)( u^2 \succ w),
\end{eqnarray*}
\noindent whereas
\begin{eqnarray*}
\succ\circ (I\otimes\; m) (u \otimes v \otimes w) &=& \succ (u \otimes v  w)\\
&=& u \succ (v  w)\\
&=& \sum_{u = u^1 u^2}( u^1 \succ v)( u^2 \succ w).
\end{eqnarray*}
\end{proof}
\ignore{
\begin{remark}\rm
$\Cal H^+$ does not admit any module-bialgebra structure \cite{mol, ckm} on $\Cal H$ given by $\vee$ or $\succ$, because neither $(\Cal H^+,\vee)$ nor $(\Cal H^+,\succ)$ is a module-coalgebra on $\Cal H$. In other words the two diagrams below are not commutative:
\diagramme{
\xymatrix{
\Cal H \otimes \Cal H \ar[d]_{\vee} \ar[rr]^{I \otimes \Delta } 
&&\Cal H \otimes\Cal H \otimes \Cal H  \ar[d]^{\Delta \otimes I \otimes I }\\
\Cal H  \ar[d]_{\Delta}&&\Cal H\otimes\Cal H \otimes \Cal H \otimes \Cal H\ar[d]^{\tau_{23} }\\
\Cal H \otimes \Cal H
&& \ar[ll]^{\vee \;\otimes \;\;\vee}\Cal H\otimes \Cal H \otimes \Cal H\otimes\Cal H}
\;\;\; \;\;\;\xymatrix{
\Cal H \otimes \Cal H \ar[d]_{\succ} \ar[rr]^{I \otimes \Delta } 
&&\Cal H \otimes\Cal H \otimes \Cal H  \ar[d]^{\Delta \otimes I \otimes I }\\
\Cal H  \ar[d]_{\Delta}&&\Cal H\otimes\Cal H \otimes \Cal H \otimes \Cal H\ar[d]^{\tau_{23} }\\
\Cal H \otimes \Cal H
&& \ar[ll]^{\succ \;\otimes \;\;\succ}\Cal H\otimes \Cal H \otimes \Cal H\otimes\Cal H}
}
\end{remark}
}
\begin{proposition}\label{vee-shu}
The action $\vee$ makes the following diagram commute:
\diagramme{
\xymatrix{
\Cal H^+ \ootimes \Cal H^+ \ar[d]_{\vee} \ar[rrdd]^{\tau_{23}\circ(\Delta' \ootimes \Delta') } 
&&\\
\Cal H^+  \ar[d]_{\Delta'}&&\\
\Cal H \otimes \Cal H^+
&& \ar[ll]^{\sshu \ootimes \vee}\Cal H\otimes \Cal H \otimes (\Cal H^+\ootimes\Cal H^+)} 
}
where $\Delta'(u):=\Delta(u)-u\otimes\un$ for any $u\in \Cal H^+$.
\end{proposition}
\begin{proof}
The compatibility of the deconcatenation with the shuffle product is written as follows:
\begin{equation}
\Delta(u\shu v)=\sum_{u=u^1u^2,\,v=v^1v^2}(u^1\shu v^1)\otimes (u^2\shu v^2)
\end{equation}
for any pure tensors $u,v\in \Cal H$. Dropping the terms with $\un$ on the right side of the tensor product yields
\begin{equation}\label{comp-shu}
\Delta'(u\shu v)=\sum_{u=u^1u^2,\,v=v^1v^2,\,(u^2,v^2)\neq(\mathbf 1,\mathbf 1)}(u^1\shu v^1)\otimes (u^2\shu v^2).
\end{equation}
Keeping only the terms of both sides of \eqref{comp-shu} with righmost letter in $v$ gives:
\begin{equation}
\Delta'(u\vee v)=\sum_{u=u^1u^2,\,v=v^1v^2,\,(u^2,v^2)\neq(\mathbf 1,\mathbf 1)}(u^1\shu v^1)\otimes (u^2\vee v^2),
\end{equation}
which proves Proposition \ref{vee-shu}.
\end{proof}
\begin{remark}\rm The following diagram also commutes:
\diagramme{
\xymatrix{
\Cal H^+ \ootimes \Cal H^+ \ar[d]_{\prec} \ar[rrdd]^{\tau_{23}\circ(\Delta'' \ootimes \Delta'') } 
&&\\
\Cal H^+  \ar[d]_{\Delta''}&&\\
\Cal H^+ \otimes \Cal H
&& \ar[ll]^{\prec \ootimes \sshu}(\Cal H^+\ootimes\Cal H^+)\otimes\Cal H\otimes \Cal H } 
}
\noindent where $\Delta''(u):=\Delta(u)-\un\otimes u$ for any $u\in \Cal H^+$. The proof is similar to the proof of Proposition~\ref{vee-shu}.
\end{remark}
%%%%%%%%%%%%%%%%%%%%%%%%%%%%%%%%%%%%%%%%%%%%%%%%%%%%%%%%%%%%%%%%%%%
%%%%%%%%%%%%%%%%%%%%%%%%%%%%%%%%%%%%%%%%%%%%%%%%%%%%%%%%%%%%%%%%%%%
%%%%%%%%%%%%%%%%%%%%%%%%%%%%%%%%%%%%%%%%%%%%%%%%%%%%%%%%%%%%%%%%%%%
%%%%%%%%%%%%%%%%%%%%%%%%%%%%%%%%%%%%%%%%%%%%%%%%%%%%%%%%%%%%%%%%%%%%%%%%%%%%%%%
%%%%%%%%%%%%%%%%%%%%%%%%%%%%%%%%%%%%%%%%%%%%%%%%%%%%%%%%%%%%%%%%%%%%%%%%%%%%%%%
%%%%%%%%%%%%%%%%%%%%%%%%%%%%%%%%%%%%%%%%%%%%%%%%%%%%%%%%%%%%%%%%%%%
\begin{remark}\rm
The infinitesimal bialgebra structure on $(\Cal H,m,\Delta)$  (in the category of vector spaces) does not give rise to an infinitesimal bialgebra structure in the category of $(\Cal H, \shu)$-modules, because $\Cal H$ is not a module-algebra on $\Cal H$. In other words the diagram below is \textsl{not} commutative:
\diagramme{
\xymatrix{
\Cal H \otimes \Cal H \otimes \Cal H \ar[d]_{I\otimes\; m} \ar[rr]^{\Delta  \otimes I \otimes I} 
&&\Cal H \otimes\Cal H \otimes \Cal H \otimes \Cal H \ar[d]^{\tau_{23} }\\
\Cal H \otimes\Cal H \ar[d]_{\sshu}&&\Cal H\otimes\Cal H \otimes \Cal H\otimes\Cal H\ar[d]^{\sshu\otimes \sshu }\\
\Cal H 
&& \ar[ll]^{m}\Cal H\otimes \Cal H }
}
\end{remark}
\noindent Let us finally make a simple restatement of Corollary \ref{theo3}.
\begin{proposition}
The two following diagrams are commutative:
\diagramme{
\xymatrix{
\Cal H \otimes \Cal H^+ \otimes \Cal H^+ \ar[d]_{I\otimes\; m} \ar[rr]^{\Delta  \otimes I \otimes I} 
&&\Cal H \otimes\Cal H \otimes \Cal H^+ \otimes \Cal H^+ \ar[d]^{\tau_{23} }\\
\Cal H \otimes\Cal H^+ \ar[d]_{\sshu}&&\Cal H\otimes\Cal H^+ \otimes \Cal H\otimes\Cal H^+\ar[d]^{\vee \;\otimes\;\; \sshu }\\
\Cal H^+ 
&& \ar[ll]^{m}\Cal H^+\otimes \Cal H^+ }\hskip 8mm
\xymatrix{
\Cal H \otimes \Cal H^+ \otimes \Cal H^+ \ar[d]_{I\otimes\; m} \ar[rr]^{\Delta  \otimes I \otimes I} 
&&\Cal H \otimes\Cal H \otimes \Cal H^+ \otimes \Cal H^+ \ar[d]^{\tau_{23} }\\
\Cal H \otimes\Cal H^+ \ar[d]_{\sshu}&&\Cal H\otimes\Cal H^+ \otimes \Cal H\otimes\Cal H^+\ar[d]^{\sshu\otimes  \succ }\\
\Cal H^+ 
&& \ar[ll]^{m}\Cal H^+\otimes \Cal H^+ }
}

which means that the concatenation $m$ is a morphism of left $\Cal H^+$-modules, where $\Cal H$ acts on $\Cal H^+$ by the shuffle product $\shu$, and $\Cal H$ acts on $\Cal H^+\otimes \Cal H^+$ either by $(\vee\otimes\shu)\circ\tau_{23}\circ (\Delta\otimes I\otimes I)$ or by $(\shu\otimes\succ)\circ\tau_{23}\circ (\Delta\otimes I\otimes I)$.
\end{proposition}
%%%%%%%%%%%%%%%%%%%%%%%%%%%%%%%%%%%%%%%%%%%%%%%%%%%%%%%%%%%%%%%%%%%
%%%%%%%%
%%%%%%%%%%%%%%%%%%%%%%%%%%%%%%%%%%%%%%%%%%%%%%%%%%%%%%%%%%%%%%%%%%%

\end{document}